\documentclass{article}
\usepackage{amsmath}
\usepackage{amssymb}
\usepackage{amsthm}
\usepackage{bbm,mathtools}

\newcommand{\op}[1]{\prescript{o\!}{}{#1}}

\newcommand{\one}{\mathbbm 1}

\def\gph{\mathop{\rm gph}}

\def\reals{\mathbb{R}}
\def\naturals{\mathbb{N}}
\def\uball{\mathbb{B}}
\def\ereals{\overline{\mathbb{R}}}

\def\cone{\mathop{\rm cone}}

\def\supp{\mathop{\rm supp}}

\def\comp{\raise 1pt \hbox{$\scriptstyle\circ$}}

\def\argmin{\mathop{\rm argmin}\limits}

\def\minimize{\mathop{\rm minimize}\limits}
\def\maximize{\mathop{\rm maximize}\limits}
\def\esssup{\mathop{\rm ess\ sup}\nolimits}

\def\st{\mathop{\rm subject\ to}}

\def\dom{\mathop{\rm dom}\nolimits}

\def\naturals{\mathbb{N}}
\def\upto{{\raise 1pt \hbox{$\scriptstyle \,\nearrow\,$}}}
\def\downto{{\raise 1pt \hbox{$\scriptstyle \,\searrow\,$}}}

\def\cl{\mathop{\rm cl}\nolimits}
\def\co{\mathop{\rm co}}

\def\epi{\mathop{\rm epi}}

\def\tos{\rightrightarrows}
\def\FF{(\F_t)_{t\ge 0}}
\def\ovr{\mathop{\rm over}}

\def\A{{\cal A}}
\def\B{{\cal B}}
\def\C{{\cal C}}
\def\cD{{\textrm(D)}}

\def\F{{\cal F}}

\def\L{{\cal L}}
\def\M{{\cal M}}
\def\N{{\cal N}}

\def\R{{\cal ﬂR}}
\def\S{{\cal S}}

\def\T{{\cal T}}
\def\U{{\cal U}}

\def\Y{{\cal Y}}

\newtheorem{theorem}{Theorem}
\newtheorem{lemma}[theorem]{Lemma}
\newtheorem{corollary}[theorem]{Corollary}
\newtheorem{proposition}[theorem]{Proposition}
\newtheorem{example}{Example}
\newtheorem{remark}{Remark}
\theoremstyle{definition}
\newtheorem{definition}{Definition}

\theoremstyle{empty}

\begin{document}
\title{Convex integral functionals of regular processes}
\author{Teemu Pennanen\thanks{Department of Mathematics, King's College London,
Strand, London, WC2R 2LS, United Kingdom} \and 
Ari-Pekka Perkki\"o\thanks{Department of Mathematics, Technische Universit\"at Berlin, Building MA,
Str. des 17. Juni 136, 10623 Berlin, Germany. The author is grateful to the Einstein Foundation for the financial support.}}

\maketitle

\begin{abstract}
This article gives dual representations for convex integral functionals on the linear space of regular processes. This space turns out to be a Banach space containing many more familiar classes of stochastic processes and its dual can be identified with the space of optional Radon measures with essentially bounded variation. Combined with classical Banach space techniques, our results allow for a systematic treatment of stochastic optimization problems over BV processes and, in particular, yields a maximum principle for a general class of singular stochastic control problems.
\end{abstract}


\section{Introduction}

This article studies convex integral functionals of the form
\[
EI_h(v) = E\int_0^Th_t(v_t)d\mu_t
\]
defined on the linear space $\R^1$ of regular processes in a filtered probability space $(\Omega,\F,\FF,P)$. Here $\mu$ is a positive optional measure on $[0,T]$ and $h$ is a convex normal integrand on $\Omega\times[0,T]\times\reals^d$. An optional cadlag process $v$ of class $\cD$ is {\em regular} if $Ev_{\tau^\nu}\to Ev_\tau$ for every increasing sequence of stopping times $\tau^\nu$ converging to a finite stopping time $\tau$ or equivalently (see \cite[Remark~50d]{dm82}), if the predictable projection and the left limit of $v$ coincide. Regular processes is quite a large family of stochastic processes containing e.g.\ continuous adapted processes, Levy processes and Feller processes as long as they are of class $(D)$. A semimartingale is regular if and only if it is of class $\cD$ and the predictable BV part of its Doob--Meyer decomposition is continuous.

Inspection of \cite{bis78} reveals that $\R^1$ is a Banach space under a suitable norm and its dual may be identified with the space $\M^\infty$ of optional random measures with essentially bounded variation. Our main result characterizes the corresponding conjugate and subdifferential of $EI_h$ under suitable conditions on the integrand $h$. Our main result applies, more generally, to functionals of the form $EI_h+\delta_{R^1(D)}$, where $\R^1(D)$ denotes the convex set of regular processes that, outside an evanecent set, take values in $D_t(\omega):=\cl\dom h_t(\cdot,\omega)$. Here, as usual, $\delta_{R^1(D)}$ is the {\em indicator function} of $\R^1(D)$ taking the value $0$ on $\R^1(D)$ and $+\infty$ outside of $\R^1(D)$.

Our main result allows for functional analytic treatment of various stochastic optimization problems where one minimizes an integral functional over the space of BV-processes. Our original motivation came from mathematical finance where BV-processes arise naturally as trading strategies in the presence of transaction costs. In this paper, we give an application to singular stochastic control by deriving a dual problem and a maximum principle for a fairly general class of singular control problems and extends and unifies singular control models of e.g.\ \cite{bs77,elk81,bsw80,ls86}. Applications to mathematical finance will be given in a separate article.

This paper combines convex analysis with the general theory of stochastic processes. More precisely, we employ the duality theory of integral functionals on the space of continuous functions developed by \cite{roc71} combined with Bismut's characterization of regular processes as optional projections of continuous stochastic processes; see \cite{bis78}. Our main result states that if the conjugate $h^*$ of $h$ is the optional projection of a convex normal integrand that allows for Rockafellar's dual representation of $I_h$ scenariowise, then under mild integrability conditions, the dual representation of $EI_h+\delta_{\R^1(D)}$ is given simply as the expectation of that of $I_h+\delta_{C(D)}$, where $C(D)$ denotes the continuous selections of $D$. The proof is more involved than the classical results on integral functionals on decomposable spaces or on spaces of continuous functions. To treat the space of regular processes, techniques from both cases need to be combined in a nontrivial way. Our proof is based on recent results on optional projections of normal integrabs from \cite{kp16} and conjugate results for continuous functions from \cite{per17}.


\section{Integral functionals and duality}\label{sec:if}

This section collects some basic facts about integral functionals defined on the product of a measurable space $(\Xi,\A)$ and a Suslin locally convex vector space $U$. In the applications below, $\Xi$ is either $\Omega$, $[0,T]$ or $\Omega\times[0,T]$. Recall that a Hausdorff topological space is {\em Suslin} if it is a continuous image of a complete separable metric space. We will also assume that $U$ is a countable union of Borel sets that are Polish spaces in their relative topology. Examples of such spaces include separable Banach spaces as well as their topological duals when equipped with the weak$^*$-topology. Indeed, such dual spaces are Suslin \cite[Proposition A.9]{tre67} and their closed unit balls are metrizable in the weak$^*$-topology by \cite[Theorem V.5.1]{ds88}, compact by the Banach--Alaoglu theorem, and thus separable by \cite[Theorem I.6.25]{ds88}.

A set-valued mapping $S:\Xi\tos U$ is {\em measurable} if the inverse image $S^{-1}(O):=\{\xi\in\Xi\,|\,S(\xi)\cap O\ne\emptyset\}$ of every open $O\subseteq S$ is in $\A$. An extended real-valued function $f:U\times\Xi\to\ereals$ is said to be a {\em normal integrand} if the {\em epigraphical mapping}
\[
\xi\mapsto\epi f(\cdot,\xi):=\{(u,\alpha)\in U\times\reals|\, f(u,\xi)\leq\alpha\}
\]
is closed-valued and measurable. 
A normal integrand $f$ is said to be {\em convex} if $f(\cdot,\xi)$ is a convex function for every $\xi\in\Xi$. A normal integrand is always $\B(U)\otimes\A$-measurable, so $\xi\mapsto f(u(\xi),\xi)$ is $\A$-measurable whenever $u:\Xi\to U$ is $\A$-measurable. Conversely, if $(\Xi,\A)$ is complete with respect to some $\sigma$-finite measure $m$, then any $\B(U)\otimes\A$-measurable function $f$ such that $f(\cdot,\xi)$ is lsc, is a normal integrand; see Lemma~\ref{lem:nor} in the appendix. Note, however, that the optional $\sigma$-algebra on $\Omega\times[0,T]$ is not complete \cite{ran90}, so we cannot always use this simple characterization when studying integral functionals of optional stochastic processes.

Given a normal integrand and a nonnegative measure $m$ on $(\Xi,\A)$ the associated {\em integral functional} 
\[
I_f(u):=\int_\Xi f(u(\xi),\xi)dm(\xi)
\]
is a well-defined extended real-valued function on the space $L^0(\Xi,\A,m;U)$ of equivalence classes of $U$-valued $\A$-measurable functions. Here and in what follows, we define the integral of a measurable function as $+\infty$ unless the positive part of the function is integrable. This convention is not arbitrary but specifically suited for studying minimization problems involving integral functionals. The function $I_f$ is called the {\em integral functional} associated with the normal integrand $f$. If $f$ is convex, $I_f$ is a convex function on $L^0(\Xi,\A,m;U)$.

Normal integrands are quite general objects and they arise naturally in various applications. We list below some useful rules for checking whether a given function is a normal integrand. When $U$ is a Euclidean space, these results can be found e.g.\ in \cite{roc76,rw98}. For Suslin spaces, we refer to the appendix. 

A function $f:U\times\Xi\rightarrow\reals$ is a {\em Carath\'eodory integrand} if $f(\cdot,\xi)$ is continuous for every $\xi$ and $f(u,\cdot)$ is measurable for every $u\in U$. Caratheodory integrands are normal; see Proposition~\ref{prop:cni} in the appendix. If $S:\Xi\tos U$ is a measurable closed-valued mapping then its {\em indicator function}
\[
\delta_S(u,\xi) :=
\begin{cases}
0\quad&\text{if }u\in S(\xi)\\
+\infty\quad&\text{otherwise},
\end{cases}
\]
is a normal integrand.

Many algebraic operations preserve normality. In particular, pointwise sums, recession functions and conjugates of proper normal integrands are again normal integrands; see Lemma~\ref{lem:oni} in the appendix. Recall that the {\em recession function} of a closed proper convex function $g$ is defined by
\[
g^\infty(u)=\sup_{\alpha>0}\frac{g(\bar u+\alpha u)-g(\bar u)}{\alpha},
\]
where the supremum is independent of the choice of $\bar u$ in the {\em domain} \[
\dom g:=\{u\,|\,g(u)<\infty\}
\]
of $g$; see \cite[Corollary~3C]{roc66}.

When $U$ is in separating duality with another linear space $Y$, the {\em conjugate} of $g$ is the extended real valued function $g^*$ on $Y$ defined by
\[
g^*(y) = \sup_{u\in U}\{\langle u,y\rangle - g(u)\}.
\]
In particular, the conjugate of the indicator function $\delta_S$ of a set $S\subset U$ is the {\em support function}
\[
\sigma_S(y):=\sup_{u\in S}\langle u,y\rangle
\]
of $S$. If $S$ is a cone, then $\sigma_S=\delta_{S^*}$, where 
\[
S^*:=\{y\in Y\,|\,\langle u,y\rangle\le 0\ \forall u\in S\},
\]
the {\em polar cone} of $S$. When $g$ is closed and proper, the {\em biconjugate theorem} says that $g=g^{**}$. This implies, in particular, that if $g$ is closed and proper its recession function can be expressed as
\begin{equation}\label{recsup}
g^\infty=\delta_{\dom g^*}^*.
\end{equation}

\subsection{Integral functionals on decomposable spaces}\label{sec:dec}

A space $\U\subseteq L^0(\Xi,\A,m;U)$ is {\em decomposable} if
\[
\one_Au+\one_{\Xi\setminus A}u'\in\U
\]
whenever $u\in\U$, $A\in\A$ and $u'\in L^0(\Xi,\A,m;U)$ is such that the closure of the range of $u'$ is compact. The following result combines the results of Rockafellar~\cite{roc68,roc71a} with their reformulation to Suslin spaces by Valadier~\cite{val75}.

\begin{theorem}[Interchange rule]\label{thm:icr}
Assume that $U=\reals^d$ or that $\A$ is $m$-complete. Given a normal integrand $f$ on $U$, we have
\[
\inf_{u\in\U}I_f(u) = \int_\Xi\inf_{u\in U}f(u,\xi)dm(\xi)
\]
as long as the left side is less than $+\infty$.
\end{theorem}

%

The interchange rule is convenient for calculating conjugates of integral functionals on decomposable spaces. Assume that $Y$ is a Suslin space in separating duality with $U$ and assume that $\Y\subseteq L^0(\Xi,\A,m;Y)$ is a decomposable space in separating duality with $\U$ under the bilinear form
\[
\langle u,y\rangle := \int_\Xi \langle u(\xi),y(\xi)\rangle dm(\xi).
\]
The first part of the following theorem is Valadier's extension of Rockafellar's conjugation formula to Suslin-valued function spaces; see~\cite{val75}. For a convex function $g$ on $U$, $y\in Y$ is a {\em subgradient} of $g$ at $u$ if
\[
g(u')\ge g(u) + \langle u'-u,y\rangle\quad\forall u'\in U.
\]
The set $\partial g(u)$ of all subgradients is known as the {\em subdifferential} of $g$ at~$u$. We often use the fact $y\in\partial g(u)$ if and only if
\[
g(u)+g^*(y)=\langle u,y\rangle.
\]
For a normal integrand $f$ and for any $u\in L^0(\Xi,\A,m;U)$, we denote by $\partial f(u)$ the set-valued mapping $\xi\mapsto\partial f(u(\xi),\xi)$, where the subdifferential is taken with respect to the $u$-argument.

\begin{theorem}\label{thm:cif}
Assume that $U=\reals^d$ or that $\A$ is $m$-complete. Given a normal integrand $f$ on $U$, the integral functionals $I_f$ and $I_{f^*}$ on $\U$ and $\Y$ are conjugates of each other as soon as they are proper and then $y\in \partial I_f(u)$ if and only if
\begin{align*}
y\in \partial f(u)\quad m\text{-a.e.}
\end{align*}
while $(I_f)^\infty=I_{f^\infty}$.
\end{theorem}

\begin{proof}
The first claim is the main theorem of \cite{val75}. When $I_f$ and $I_{f^*}$ are conjugates of each other, then $y\in\partial I_f(u)$ if and only if
\[
I_f(u)+I_{f^*}(y)=\langle u,y\rangle
\]
which, by the Fenchel inequality $f(u)+f^*(y)\ge \langle u,y\rangle$, is equivalent to 
\[
f(u)+f^*(y)=\langle u,y\rangle\quad m\text{-a.e.}
\]
which in turn means that $y\in\partial f(u)$ $m$-almost everywhere.

To prove the recession formula, let $\bar u\in\dom I_f$ and $\bar y\in\dom I_{f^*}$. We have $f(u,\xi)\ge \langle u,\bar y(\xi)\rangle -f^*(\bar y(\xi),\xi)$ so the function $f(\bar u+ u)-f(\bar u)$ has an integrable lower bound. By convexity, the difference quotient
\[
\frac{f(\bar u+\alpha u)-f(\bar u)}{\alpha}
\]
is nondecreasing in $\alpha$, so monotone convergence theorem gives
\begin{align*}
(I_{f^\infty})(u) &= \int_\Xi\lim_{\alpha\upto\infty}\frac{f(\bar u+\alpha u)-f(\bar u)}{\alpha}dm\\
&= \lim_{\alpha\upto\infty}\int_\Xi\frac{f(\bar u+\alpha u)-f(\bar u)}{\alpha}dm\\
&= \lim_{\alpha\upto\infty}\frac{I_f(\bar u+\alpha u)-I_f(\bar u)}{\alpha}\\
&=(I_f)^\infty(u),
\end{align*}
where the first and the last equation hold since $f(\cdot,\xi)$ and $I_f$ are lower semicontinuous; see \cite[Corollary~3C]{roc66}.
\end{proof}

\subsection{Integral functionals of continuous functions}\label{sec:cont}

Consider now the case where $U=\reals^d$ equipped with the Euclidean topology and $\Xi$ is a compact interval $[0,T]\subset\reals$ equipped with the Borel sigma-algebra and a nonnegative Radon measure $\mu$ with full support, i.e., $\supp \mu=[0,T]$. This section reviews conjugation of convex integral functionals on the space $C$ of $\reals^d$-valued continuous functions on an interval $[0,T]$. Recall that under the supremum norm, $C$ is a Banach space whose dual can be identified with the linear space $M$ of (signed) Radon measures $\theta$ through the bilinear form
\[
\langle v,\theta\rangle :=\int vd\theta.
\]
Here and in what follows, the domain of integration is $[0,T]$ unless otherwise specified. 

Given a normal integrand $h$ on $\reals^d\times[0,T]$, consider the integral functional functional $I_h$ on $C$. The space $C$ is not decomposable so one cannot directly apply the interchange rule to calculate conjugate of $I_h$. Rockafellar~\cite{roc71} and more recently Perkki\"o~\cite{per14,per17} gave conditions under which
\begin{equation}\label{eq:IJ}
(I_h)^*=J_{h^*},
\end{equation}
where for a normal integrand $f$ on $\reals^d\times[0,T]$, the functional $J_f:M\to\ereals$ is defined by
\[
J_f(\theta)=\int f_t((d\theta^a/d\mu)_t)d\mu_t+\int f_t^\infty((d\theta^s/d|\theta^s|)_t)d|\theta^s|_t,
\]
where $\theta^a$ and $\theta^s$ are the absolutely continuous and the singular part, respectively, of $\theta$ with respect to $\mu$ and $|\theta^s|$ is the total variation of $\theta^s$. From now on, we omit the time index and write simply
\[
J_f(\theta)=\int f(d\theta^a/d\mu)d\mu+\int f^\infty(d\theta^s/d|\theta^s|)d|\theta^s|.
\]

The validity of \eqref{eq:IJ} depends on the behavior of the set 
\[
D_t:=\cl\dom h_t
\]
as a function of $t$. Recall that a set-valued mapping $S$ from $[0,T]$ to $\reals^d$ is {\em inner semicontinuous} (isc) if $\{t \mid S_t\cap O\neq\emptyset\}$ is an open set for any open $O$; see \cite[Section~5B]{rw98}. 
We will use the notation $\partial^sh_t:=\partial\delta_{D_t}$. More explicitly, $x\in\partial^sh_t(v)$ means that $v\in D_t$ and
\[
\langle x,v'-v\rangle \le 0\quad\forall v'\in D_t,
\]
i.e.\ $\partial^sh_t(v)$ is the {\em normal cone} to $D_t$ at $v$. Given a $v\in C$, we denote the set-valued mapping $t\mapsto\partial^sh_t(v_t)$ by $\partial^s h(v)$. The following is from \cite{per17}.

\begin{theorem} \label{thm:if1}
Assuming $I_h+\delta_{C(D)}$ and $J_{h^*}$ are proper, they are conjugates of each other if and only if $\dom h$ is isc and $C(D)=\cl(\dom I_h\cap C(D))$, and then $\theta\in\partial (I_h+\delta_{C(D)})(y)$ if and only if
\begin{align}\label{eq:sd}
\begin{split}
d\theta^a/d\mu&\in\partial h(y)\quad\mu\text{-a.e.},\\
d\theta^s/d|\theta^s| &\in\partial^sh(y)\quad|\theta^s|\text{-a.e.}
\end{split}
\end{align}
\end{theorem}

The conditions of the theorem have been analyzed in \cite{per17}. The inner semicontinuity condition goes back to the continuous selection theorems of Michael \cite{mic56}. The domain condition  $C(D)=\cl(\dom I_h\cap C(D))$ holds automatically, in particular, if $h$ is an indicator function. The condition means that $\dom I_h$ is dense in $C(D)$. For example, for $h_t(v)=v/t$, $t>0$, and the Lebesgue measure, the condition is satisfied if and only if $h_0(v)=\delta_{\{0\}}(v)$.



\section{Integral functionals of continuous processes}\label{sec:ifni}

For the remainder of this paper, we fix a complete probability space $(\Omega,\F,P)$. This section studies integral functionals on the Banach space $L^1(C)$ of random continuous functions $v$ with the norm
\[
\|v\|_{L^1(C)}:=E\sup_{t\in[0,T]}|v_t|.
\]
Here and in what follows, $E$ denotes the integral with respect to $P$ (expectation). The results of this section will be used to derive our main results on integral functionals of regular processes

We endow the space $M$ of Radon measures with the Borel sigma-algebra associated with the weak*-topology and we denote by $L^\infty(M)$ the linear space of $M$-valued random variables $\theta$ with essentially bounded variation\footnote{By usual monotone class arguments, the elements of $L^\infty(M)$ are {\em random Radon measures} in the sense of \cite{dm82}.}. The {\em total variation} of a $\theta\in M$ will be denoted by $\|\cdot\|_{TV}$. The first part of the following is from \cite[Theorem~2]{bis78}.

\begin{theorem}\label{thm:banach0}
The Banach dual of $L^1(C)$ may be identified with $L^\infty(M)$ through the bilinear form
\[
\langle v,\theta\rangle := E\int v d\theta.
\]
The dual norm on $L^\infty(M)$ can be expressed as
\[
\|\theta\|_{L^\infty(M)} = \esssup\|\theta\|_{TV}.
\]
\end{theorem}

\begin{proof}
By Lagrangian duality,
\begin{align*}
\|\theta\|_{L^\infty(M)} &= \sup\{\langle \theta,v\rangle\mid E\|v\|_{C}\le 1\}\\
&=\inf_{\lambda\in\reals_+} \sup_{v}[\langle \theta,v\rangle-\lambda E\|v\|_{C}+\lambda]\\
&=\inf_{\lambda\in\reals_+}\{E\delta_{\uball}(\theta/\lambda)+\lambda\}\\
&=\esssup\|\theta\|_{TV},
\end{align*}
where $\uball$ is the closed unit ball of the total variation norm and the third equality follows from Theorem~\ref{thm:cif}.
\end{proof}

We will study integral functionals associated with normal integrands that are defined for each $\omega\in\Omega$ as integral functionals on $C$ and $M$. Both $C$ and $M$ are countable unions of Borel sets that are Polish spaces in the relative topology, so we are in the setting of Section~\ref{sec:if}.  We allow both the integrand $h$ and the measure $\mu$ to be random. More precisely, we will assume that $\mu$ is a nonnegative random Radon measure with full support almost surely and that $h$ is a convex normal integrand on $\reals^d\times\Xi$, where $\Xi=\Omega\times[0,T]$ is equipped with the product sigma-algebra $\F\otimes\B([0,T])$. We define $I_h:C\times\Omega\to\ereals$, $C(D):\Omega\tos C$ and $J_{h^*}:M\times\Omega\to\ereals$ by
\begin{align*}
I_h(v,\omega)&:=I_{h(\cdot,\omega)}(v),\\
C(D)(\omega)&:=C(D(\omega)),\\
J_{h^*}(\theta,\omega)&:= J_{h^*(\cdot,\omega)}(\theta),
\end{align*}
where the right sides are defined as in Section~\ref{sec:cont}.

\begin{lemma}\label{lem:sn}
If $(\omega,t)\mapsto D_t(\omega)$ is measurable, closed-valued and isc, then $C(D)$ is measurable and closed-valued.
\end{lemma}

\begin{proof}
That $C(D)$ is closed-valued is evident. By \cite[Theorem~1]{roc71a}, it suffices to show that $\omega\mapsto d(v,C(D(\omega)))$ is measurable for every $v\in C$. We have
\begin{align*}
d(v,C(D(\omega))) &= \inf_{w\in C(D(\omega))}\sup_{t\in[0,T]}|w_t-v_t|\\
&\ge \sup_{t\in[0,T]}\inf_{w\in C(D(\omega))}|w_t-v_t|\\
&= \sup_{t\in[0,T]}\inf_{w\in D_t(\omega)}|w-v_t|\\
&= \sup_{t\in[0,T]}d(v_t,D_t(\omega))=:r(\omega),
\end{align*}
where the third equality follows from the inner semicontinuity. On the other hand, fixing an $\epsilon>0$ and defining $S_t(\omega):=\uball_{r(\omega)+\epsilon}(v_t)$, the mapping $t\mapsto D_t(\omega)\cap S_t(\omega)$ is isc, by \cite[Theorem~??]{per17}. By Michael's selection theorem, it admits a continuous selection $\bar w$, so
\[
d(v,C(D(\omega)))\le r(\omega)+\epsilon.
\]
Since $\epsilon>0$ was arbitrary, we must have 
\[
d(v,C(D(\omega)))=\sup_{t\in[0,T]}d(v_t,D_t(\omega)).
\]
By \cite[Proposition~14.47]{rw98}, the measurability of $D$ implies that $(\omega,t)\mapsto d(v_t,D_t(\omega))$ is measurable. The measurability of $d(v,C(D(\omega)))$ now follows from \cite[Lemma~III.39]{cv77}.
\end{proof}

\begin{lemma}\label{lem:nif}
If $h(\omega)$ satisfies, for $P$-almost every $\omega$, the conditions of Theorem~\ref{thm:if1}, then $I_h+\delta_{C(D)}$ and $J_{h^*}$ are normal integrands conjugate to each other. 
\end{lemma}

\begin{proof}
By Theorem~\ref{thm:if1}, $I_h(\cdot,\omega)+\delta_{C(D(\omega)}$ and $J_{h^*}(\cdot,\omega)$ are conjugate to each other almost surely. The uniform topology on $C$ satisfies both (a) and (b) of Theorem~\ref{thm:mif} in the appendix, so $I_h$ is a normal integrand. By Lemma~\ref{lem:sn}, $\delta_{C(D)}$ is a normal integrand as well, so  $I_h+\delta_{C(D)}$ and $J_{h^*}$ are normal integrands by Lemma~\ref{lem:oni} in the appendix.
\end{proof}

By Lemma~\ref{lem:nif} and Lemma~\ref{lem:nor}, the integral functionals $EI_h:L^1(C)\to\ereals$ and $EJ_{h^*}:L^\infty(M)\to\ereals$ are well defined. An application of the interchange rule~Theorem~\ref{thm:icr} and Lemma~\ref{lem:nif} now gives expressions for the conjugate and subdifferential of $E[I_h+\delta_{C(D)}]$. Recall the notation for the subdifferential mapping from Section~\ref{sec:dec}.

\begin{theorem}\label{thm:1}
If $h(\omega)$ satisfies, for $P$-almost every $\omega$, the conditions of Theorem~\ref{thm:if1}, then the convex functions $E[I_h+\delta_{C(D)}]:L^1(C)\to\ereals$ and $EJ_{h^*}:L^\infty(M)\to\ereals$ are conjugate to each other as soon as they are proper and then $\theta\in \partial E[I_h+\delta_{C(D)}](v)$ if and only if
\begin{align*}
d\theta^a/d\mu&\in\partial h(v)\quad \mu\text{-a.e.},\\
d\theta^s/d|\theta^s| &\in \partial^sh(v)\quad|\theta^s|\text{-a.e.}
\end{align*}
almost surely.
\end{theorem}

\begin{proof}
Theorem~\ref{thm:cif} and Lemma~\ref{lem:nif} give
\begin{align*}
(E[I_h+\delta_{C(D)}])^*(\theta) &= E[I_h+\delta_{C(D)}]^*(\theta) =EJ_{h^*}(\theta).
\end{align*}
The subgradient characterization now follows from Theorems~\ref{thm:cif} and \ref{thm:if1}.

\end{proof}

\section{Regular processes}\label{sec:reg}

Let $\FF$ be an increasing sequence of $\sigma$-algebras on $\Omega$ that satisfies the usual hypotheses that $\F_t=\bigcap_{t'>t}\F_{t'}$ and $\F_0$ contains all the $P$-null sets. We denote by $\T$ the set of {\em stopping times}, that is, functions $\tau:\Omega\to[0,T]\cup\{+\infty\}$ such that $\{\tau\le t\}\in\F_t$ for all $t\in[0,T]$. A process is {\em optional} if it is measurable with respect to the $\sigma$-algebra generated by right-continuous adapted processes. If $v$ is {\em $\T$-integrable} in the sense that $v_\tau$ is integrable for every $\tau\in\T$, then, by \cite[Theorem 5.1]{hwy92}, there exists a unique (up to indistinguishability) optional process $\op v$ such that
\begin{align}\label{eq:op}
E[v_\tau\one_{\{\tau<\infty\}}\mid \F_\tau] &= \op{v}_\tau\one_{\{\tau<\infty\}}\quad P\text{-a.s. for all $\tau\in\T$}.
\end{align}
The process $\op v$ is called the {\em optional projection} of $v$. In particular, every $v\in L^1(C)$ has a unique optional projection.

We will denote by $\R^1$ the space of {\em regular processes}, i.e., the optional c\'adl\'ag\footnote{right-continuous with left limits} processes $v$ of class $\cD$ such that $Ev_{\tau^\nu}\to Ev_\tau$ for every increasing sequence of stopping times $\tau^\nu$ converging to a finite stopping time $\tau$ or equivalently (see \cite[Remark~50d]{dm82}), such that the predictable projection and the left limit of $v$ coincide. Recall that a process $v$ is of {\em class} $\cD$ if $\{v_\tau\,|\, \tau\in \T\}$ is uniformly integrable. By \cite[Theorem~3]{bis78}, the optional projection is a linear surjection of $L^1(C)$ to $\R^1$.

\begin{remark}\label{rem:regular}
An optional c\'adl\'ag process $v$ of class $\cD$ is in $\R^1$, in particular, if it is {\em quasi left-continuous} in the sense that $\lim v_{\tau^\nu}=v_\tau$ almost surely for any strictly increasing sequence of stopping times  $(\tau^\nu)_{\nu=1}^\infty$ with $\tau^\nu\nearrow\tau$. Conversely, if the filtration $\FF$ is quasi-left continuous then quasi left-continuous processes of class $\cD$ are regular; see \cite[Remark~50.(d)]{dm82} and \cite[Theorem~4.34]{hwy92}. Continuous adapted processes, Levy processes (\cite[Theorem 11.36]{hwy92}) and Feller processes (\cite[Proposition 22.20]{kal2}) are quasi left-continuous. 
A semimartingale $v$ is regular if and only if it is of class $\cD$ and has a decomposition $v=m+a$ where $m$ is a local martingale and $a$ is a continuous BV process. Indeed, a semimartingale of class $(D)$ is special so the claim follows from \cite[Remark VII.24(e)]{dm82}.
\end{remark}

We will denote by $\M^\infty\subseteq L^\infty(M)$ the space of essentially bounded {\em optional} Radon measures on $\reals^d$, i.e.\ the elements $\theta\in L^\infty(M)$ such that
\[
E\int v d\theta = E\int \op{v}d\theta\quad\forall v\in L^1(C).
\]
The following result, essentially proved already in Bismut~\cite{bis78}, shows that $\M^\infty$ may be identified with the Banach dual of $\R^1$.

\begin{theorem}\label{thm:banach}
The space $\R^1$ is a Banach space under the norm
\[
\|v\|_{\R^1} := \sup_{\tau\in\T}E|v_\tau|
\]
and its dual may be identified with $\M^\infty$ through the bilinear form
\[
\langle v,\theta\rangle_{\R^1} = E\int vd\theta.
\]
The dual norm can be expressed as 
\[
\|\theta\|_{\M^\infty} = \esssup\|\theta\|_{TV}.
\]
\end{theorem}

\begin{proof}
Since the optional projection is a surjection from $L^1(C)$ to $\R^1$, it defines a linear bijection from the quotient space $L^1(C)/K$ to $\R^1$. Here $K$ denotes the kernel of the projection. For any $v\in L^1(C)$, Jensen's inequality gives
\begin{equation}\label{eq:jensen}
\|\op{v}\|_{\R^1} = \sup_{\tau\in\T}E|E[v_\tau\,|\,\F_\tau]| \le \sup_{\tau\in\T}E|v_\tau| \le \|v\|_{L^1(C)}
\end{equation}
so the optional projection is continuous. In particular, $K$ is closed in $L^1(C)$ so $L^1(C)/K$ is a Banach space under the quotient space norm
\[
\|[v]\|_{L^1(C)/K}:=\inf_{v'\in K}\|v+v'\|_{L^1(C)}.
\]
On the other hand, for each $w\in\R^1$ and $\varepsilon>0$, \cite[Theorem~3]{bis78} gives the existence of a $v\in L^1(C)$ such that $w=\op{v}$ and $\|v\|_{L^1(C)}\le\|w\|_{\R^1}+\varepsilon$. Thus, $\|[v]\|_{L^1(C)/K}\le\|\op v\|_{\R^1}$ which together with \eqref{eq:jensen} implies that the optional projection is an isometric isomorphism from the quotient space $L^1(C)/K$ to $\R^1$. It follows that $\R^1$ is Banach and, by \cite[Proposition~2]{bis78}, its dual may be identified with $\M^\infty$. As to the dual norm,
\begin{align*}
\|\theta\|_{\M^\infty} &=\sup_{v\in\R^1}\{\langle v,\theta\rangle\,|\,\|v\|_{\R^1}\le 1\}\\
&=\sup_{v\in L^1(C)}\{\langle v,\theta\rangle\,|\,\|v\|_{L^1(C)}\le 1\},
\end{align*}
where the second equality comes from the isomorphism of $\R^1$ and $L^1(C)/K$.
\end{proof}

Theorem~\ref{thm:banach} complements the results of \cite[Section~7.1.4]{dm82} on Banach duals of adapted continuous functions and adapted c\'adl\'ag functions under the supremum norm. The dual space of adapted continuous functions consists of predictable random measures with essentially bounded variation whereas the dual of adapted c\'adl\'ag functions is given in terms of pairs of optional and predictable random measures with essentially bounded variation; see \cite[Theorem VII.67]{dm82}. In the deterministic case, Theorem~\ref{thm:banach} reduces to the familiar Riesz representation of continuous linear functionals on the space of continuous functions (the duality between $C$ and $M$).

The norm $\|\cdot\|_{\R^1}$ in Theorem~\ref{thm:banach} is studied in \cite[Section VI]{dm82}, where a general measurable process $v$ is said to be ``bounded in $L^1$'' if $\|v\|_{\R^1}<\infty$; see \cite[Definition VI.20]{dm82}. It is observed on p.~82--83 of \cite{dm82} that a sequence converging in the $\R^1$-norm has a subsequence that converges almost surely in the supremum norm. Moreover, by \cite[Theorem VI.22]{dm82}, the space of optional cadlag processes with finite $\R^1$-norm is Banach. Theorem~\ref{thm:banach} implies that regular processes form a closed subspace of this space.

\section{Integral functionals of regular processes}\label{sec:ifrp}

This section gives the main result of this paper. Given a normal integrand $h$ and a random measure $\mu$ as in Section~\ref{sec:ifni}, it characterizes the conjugate of a convex integral functional of the form
\[
EI_h(v) = E\int h(v)d\mu
\]
on the space $\R^1$ of regular processes. Note that $\R^1$ is not a decomposable space nor are the paths of a regular process continuous in general, so we are beyond the settings of Sections~\ref{sec:if} and \ref{sec:ifni}. Nevertheless, the functional $EI_h$ is well-defined on $\R^1$ since $h(v)$ is an extended real-valued measurable process for every $v\in\R^1$, so $I_h(v)$ is $\F$-measurable by Lemma~\ref{lem:op0} in the appendix.

Our main result, Theorem~\ref{thm:ifam2} below, requires some additional properties on $\mu$ and $h$. As to $\mu$, we assume that it is {\em optional}, i.e.\ that
\[
E\int vd\mu=E\int \op vd\mu
\]
for every nonnegative bounded process $v$. 
The normal integrand $h$ will be assumed ``regular'' in the sense of Definition~\ref{def:ri} below. The definition involves the notion of the optional projection of a normal integrand that we now recall; see \cite{kp16}.

A normal integrand $g$ on $\reals^d\times\Omega\times[0,T]$ is said to be {\em optional} if 
its epigraph $\epi g_t(\cdot,\omega)$ is measurable with respect to the optional sigma algebra on $\Omega\times[0,T]$. If $g$ is a convex normal integrand such that $g^*(v)^+$ is $\T$-integrable for some $\T$-integrable $v$ then, by \cite[Theorem~6]{kp16}, there exists a unique optional convex normal integrand $\op g$ such that
\begin{align}\label{eq:opni}
\op g(x)=\op[g(x)]
\end{align}
for every bounded optional process $x$. Here we use the notion of optional projection of an extended real-valued process; see the appendix. The normal integrand $\op g$ is called the {\em optional projection} of $g$. Clearly, an optional normal integrand is the optional projection of itself. In the linear case where $g(x,\omega)=v_t(\omega)\cdot x$ for a measurable $\T$-integrable process $v$, we simply have $\op g_t(x,\omega)=\op v_t(\omega)\cdot x$.

We will use the abbreviation a.s.e.\ for ``$P$-almost surely everywhere on $[0,T]$'', that is, outside an evanescent set. 
\begin{definition}\label{def:ri}
An optional convex normal integrand $h$ on $\reals^d$ is {\em regular} if $h^*=\op{\tilde h^*}$ for a convex normal integrand $\tilde h$ such that 
$\tilde h(\omega)$ satisfies, for $P$-almost every $\omega$, the conditions of Theorem~\ref{thm:if1} and
\begin{align*}
\tilde h(v) &\ge v\cdot \bar x-\alpha\ \text{ a.s.e.}\\
\tilde h^*(x) &\ge \bar v\cdot x-\alpha\ \text{ a.s.e.}
\end{align*}
for some $\bar v\in L^1(C)$ with $\bar v\in C(D)$ almost surely, optional $\bar x$ with $\int|\bar x|d\mu\in L^\infty$ and some $\T$-integrable $\alpha$ with $\int|\alpha|d\mu\in L^1$. 
\end{definition}

Before commenting on Definition~\ref{def:ri}, we state the main result of this paper, which characterizes the conjugate and the subdifferential of an integral functional on $\R^1$. Since $\R^1$ is not decomposable, we cannot directly apply the interchange rule in Theorem~\ref{thm:icr}. Instead, the idea is to apply the interchange rule to $I_{\tilde h}$ on $L^1(C)$ and to use properties of optional projections of normal integrands from \cite{kp16}. The proof is given in the appendix.

\begin{theorem}\label{thm:ifam2}
If $h$ is a regular convex normal integrand, then $EI_h+\delta_{\R^1(D)}:\R^1\rightarrow\ereals$ and $EJ_{h^*}:\M^\infty\rightarrow\ereals$ are proper and conjugate to each other and, moreover, $\theta\in\partial(EI_h+\delta_{\R^1(D)})(v)$ if and only if
\begin{align*}
d\theta^a/d\mu&\in\partial h(v)\quad \mu\text{-a.e.},\\
d\theta^s/d|\theta^s| &\in \partial^sh(v)\quad|\theta^s|\text{-a.e.}
\end{align*}
almost surely.
\end{theorem}

In the deterministic case, Theorem~\ref{thm:ifam2} gives sufficiency for Theorem~\ref{thm:if1}. Indeed, we then have $\R^1=C$ and one can simply take $\tilde h=h$ in Definition~\ref{def:ri}. Note that, in general, the assumptions in the Theorem~\ref{thm:ifam2} do not imply that $J_{h^*}$ is a normal integrand on $ M\times\Omega$. For example, in the linear case $h^*_t(x,\omega)=v_t(\omega)\cdot x$, the function $J_{h^*}(\cdot,\omega)$ is $\sigma(M,C)$-lower semicontinuous on $M$ if and only if $v$ has continuous paths (since $C$ is the topological dual of $M$ under the weak* topology $\sigma(M,C)$). 

Theorem~\ref{thm:ifam2} simplifies when $h$ is real-valued.

\begin{corollary}
Let $h$ be a real-valued optional convex normal integrand such that $EI_h$ is finite on $L^1(C)$ and
\begin{align*}
h(v) &\ge v\cdot \bar x-\alpha,\\
h^*(x) &\ge \bar v\cdot x-\alpha
\end{align*}
for some $\bar v\in L^1(C)$, optional $\bar x$ with $\int|\bar x|d\mu\in L^\infty$ and some $\T$-integrable $\alpha$ with $\int|\alpha|d\mu\in L^1$. Then $EI_h:\R^1\rightarrow\ereals$ and $EJ_{h^*}:\M^\infty\rightarrow\ereals$ are proper and conjugate to each other and, moreover, $\theta\in\partial EI_h(v)$ if and only if $\theta\ll \mu$
\begin{align*}
d\theta^a/d\mu&\in\partial h(v)\quad \mu\text{-a.e.}
\end{align*}
almost surely. Moreover, $EI_h$ is continuous throughout $\R^1$.
\end{corollary}

\begin{proof}
Our assumptions imply that $I_h$ is finite on $C$ almost surely. Thus we may choose $\tilde h =h$ in Definition~\ref{def:ri}, so $h$ is regular. The first part thus follows from Theorem~\ref{thm:ifam2} and the fact that $\partial^sh(v)=\{0\}$ for a finite $h$. It remains to show that  $EI_h$ is continuous. As in the proof of Theorem~\ref{thm:ifam2} in the appendix, we see that $EI_h(\op v)\le EI_h(v)$ for every $v\in L^1(C)$. Thus, $EI_h$ is finite on $\R^1$, so the continuity follows from \cite[Corollary 8B]{roc74} since $\R^1$ is Banach.
\end{proof}

We say that a measurable closed convex-valued mapping $S$ is {\em regular} if its indicator function $\delta_S$ is regular in the sense of Definition~\ref{def:ri}. In particular, a convex-valued isc optional mapping $S$ that admits an $L^1(C)$ selection is regular since then one can take $\tilde h=\delta_S$ in Definition~\ref{def:ri}. When $h=\delta_S$, Theorem~\ref{thm:ifam2} can be stated in terms of the normal integrand defined pointwise by $\sigma_{S_t}(x,\omega):=\sup_{v\in S_t(\omega)}x\cdot v$.




\begin{corollary}\label{cor:ind}
Let $S$ be a regular set-valued mapping. Then 
\[
\S = \{v\in\R^1\,|\,v\in S\ a.s.e.\}
\]
is closed in $\R^1$,
\[
\sigma_{\S}(\theta) = E\int \sigma_{S}(d\theta/d|\theta|)d|\theta|
\]
and $\theta\in N_{\S}(v)$ if and only if $d\theta/d|\theta|\in N_S(v)$ $|\theta|$-a.e.\ almost surely. In particular, if $S$ is cone-valued, then $\S$ is a closed convex cone and $\theta\in\S^*$ if and only if $d\theta/d|\theta|\in S^*$ $|\theta|$-almost everywhere $P$-almost surely.
\end{corollary}


Regular set-valued mappings can be characterized in terms of optional projections of set-valued mappings. If $S$ admits a $\T$-integrable a.s.e.\ selector then, by \cite[Theorem~10]{kp16}, there is a unique optional closed convex-valued mapping $\op S$ such that $\sigma_{\op S}=\op\sigma_S$. The mapping $\op S$ is called the {\em optional projection} of $S$. 
Note that if $S_t(\omega)=\{v_t(\omega)\}$ for a $\T$-integrable process $v$, we simply have $\op S_t(\omega)=\{\op v_t(\omega)\}$. 


\begin{lemma}\label{lem:rsm}
A closed convex-valued measurable mapping $S$ is regular if and only if it is the optional projection of a closed convex-valued measurable isc mapping $\tilde S$ that admits an $L^1(C)$ a.s.e.\ selection. In particular, a single-valued mapping $S_t(\omega)=\{v_t(\omega)\}$ is regular if and only if $v$ is a regular process.
\end{lemma}

\begin{proof}
If such an $\tilde S$ exists, one can take $\tilde h=\delta_{\tilde S}$ in Definition~\ref{def:ri}. To prove the necessity, let $\tilde h$ be a convex normal integrand in Definition~\ref{def:ri} so that $\op {\tilde h^*}=\sigma_S$. By \cite[Theorem 7]{kp16}, $\op[(\tilde h^*)^\infty]=[\op (\tilde h^*)]^\infty$. Since $\sigma_S$ is positively homogeneous, we get $\op[(\tilde h^*)^\infty]=\sigma_S$. On the other hand, by \eqref{recsup}, $(\tilde h^*)^\infty=\sigma_{\cl\dom {\tilde h}}$, so we may choose $\tilde S=\cl\dom {\tilde h}$.

Consider now the single-valued case and let $\tilde v\in L^1(C)$ be a selection of $\tilde S$. We have $\sigma_{\tilde S_t(\omega)}(x)\ge x\cdot\tilde v_t(\omega)$, so $x\cdot v_t(\omega)=\op\sigma_{\tilde S_t(\omega)}(x)\ge x\cdot\op{\tilde v}_t(\omega)$. Since this holds for all $x\in\reals^d$, we must have $v=\op{\tilde v}$.
\end{proof}

\begin{example}
A set-valued mapping is regular if it is a "martingale" in the sense that is the projection of a pathwise constant set-valued mapping that admits an $L^1(C)$ a.s.e. selection. Set-valued martingales in discrete time have been analyzed, e.g., in \cite{hu77,he2}.
\end{example}

\begin{example}
Sets of the form $\S^*$ in the last part of Corollary~\ref{cor:ind} are used to describe financial markets in \cite[Section 3.6.3]{ks9}, where it is assumed that
\[
S_t(\omega):=\co\cone\{\zeta^k_t(\omega)\mid k\in\naturals\}
\]
for a countable family $(\zeta^k)_{k\in\naturals}$ of adapted continuous processes such that for each $\omega$ and $t$ only a finite number of the vectors $\zeta^k_t(\omega)$ is nonzero. Such an $S$ is automatically optional and isc and thus regular. Indeed, given a family of isc mappings $(\Gamma_\alpha)$ their pointwise union is isc since $(\bigcup_{\alpha}\Gamma_\alpha)^{-1}(O)=\bigcup_\alpha((\Gamma_\alpha)^{-1}(O))$ is open for any open $O$. Thus, $S$ is isc and, by \cite[Proposition 14.11 and Exercise 14.12]{rw98}, it is also optional (even predictable).

It is clear from the above argument that the assumption, that only a finite number of the generators is nonzero, is not needed for regularity. Indeed, the mapping
\[
S_t(\omega):=\cl\co\cone\{\zeta^k_t(\omega)\mid k\in\naturals\}
\]
is still regular. More generally, $S$ is regular if is the optional projection of 
\[
\tilde S_t(\omega):=\cl\co\cone\{\tilde \zeta^k_t(\omega)\mid k\in\naturals\}
\]
for a countable family $(\zeta^k)_{k\in\naturals}$ of (non-adapted) continuous processes. Corollary~\ref{cor:3} gives an extension to nonconical models. In discrete time, such models have been studied in \cite{pp10}.
\end{example}

Applying Theorem~\ref{thm:ifam2} to the case where $h$ is support function of a closed convex-valued mapping gives the following.

\begin{corollary}\label{cor:3}
Let $S$ be an optional closed convex-valued mapping such that $\sigma_S$ is regular. Then
\[
\C=\{\theta\in\M^\infty\mid d\theta^a/d\mu\in S\ \mu\text{-a.e.},\ d\theta^s/d|\theta^s|\in S^\infty\ |\theta^s|\text{-a.e.\ $P$-a.s.}\}
\]
is closed in $\M^\infty$ and its support function has the representation
\[
\sigma_\C(v) = E\int\sigma_S(v)d\mu + \delta_{\R^1(D)},
\]
where $D_t(\omega)=\cl\dom\sigma_{S_t(\omega)}$.
\end{corollary}




\section{Maximum principle in singular stochastic control}\label{sec:app}

The space $M$ of Radon measures may be identified with the space $X_0$ of $\reals^d$-valued left-continuous functions of bounded variation on $\reals_+$ which are constant on $(T,\infty]$ and $x_0=0$. Indeed, for every $x\in X_0$, there exists a unique $Dx\in M$ such that $x_t=Dx([0,t))$ for all $t\in \reals$.   Thus $x\mapsto Dx$ defines a linear isomorphism between $X_0$ and $M$. The value of $x$ for $t>T$ will be denoted by $x_{T+}$. Similarly, the space $\M^\infty$  may be identified with the space $\N_0^\infty$ of adapted processes $x$ with $x\in X_0$ almost surely and $Dx\in\M^\infty$. 

Let $g$ and $h$ be optional normal integrands on $\reals^d$ and consider the stochastic control problem
\begin{equation*}
\begin{aligned}
&\minimize_{c \in\L^\infty}\quad & & E[I_g(x) + e(x_T) + I_{h^*}(u)]\\
&\st & &
x_t=\int_0^tAxd\mu + \int_0^tBud\mu + W_t,
\end{aligned}
\end{equation*}
where $A,B\in\reals^{d\times d}$ and $W\in L^1(C)$ is optional. Here and in what follows, $\L^p$ denotes the optional elements of the space $L^p(\Omega\times[0,T],\F\otimes\B([0,T]),\nu;\reals^d)$, where $\nu(A):=E\mu(A)$. Denoting $z_t:=\int_0^txd\mu$ and $c_t:=\int_0^tud\mu$, we can write the problem as
\begin{equation*}
\begin{aligned}
&\minimize_{c \in\A_0^\infty}\quad &  E[&I_g(\dot z) + e(\dot z_T) + I_{h^*}(\dot c)]\\
&\st & &
\begin{cases}
z_0=0,\\
\dot z_t = Az_t + Bc_t + W_t,
\end{cases}
\end{aligned}
\end{equation*}
where $\A^\infty_0$ denotes the space of optional processes with $\mu$-absolutely continuous paths and essentially bounded variation. The {\em singular control} problem is obtained by allowing $c$ to be of bounded variation, not just absolutely continuous. The problem becomes
\begin{equation}\tag{SCP}\label{scp}
\begin{aligned}
&\minimize_{c \in\N_0^\infty}\quad &  E[&I_g(\dot z) + e(\dot z_T) + J_{h^*}(Dc)]\\
&\st & &
\begin{cases}
z_0=0,\\
\dot z_t=Az_t + Bc_t + W_t,
\end{cases}
\end{aligned}
\end{equation}
where the functional $J_{h^*}:\M^\infty\to\ereals$ is defined as in Section~\ref{sec:ifrp}.

In the one-dimensional case, with $g(z)=\frac{1}{2}r|z|^2$ and
\[
h^*(c) = 
\begin{cases}
c^2 & \text{if $|c|\le k/2$},\\
k|c|-k^2/4 & \text{if $|c|\ge k/2$}
\end{cases}
\]
for some nonnegative constants $r$ and $k$, we recover a finite-horizon version of the singular stochastic control problem studied by Lehoczky and Shreve~\cite{ls86} (note that they wrote the problem in terms of the variables $x=\dot z$). Whereas \cite{ls86} analyzed the Hamilton--Jacobi--Bellman equation associated with the above one-dimensional case, we use convex duality to derive a dual problem and optimality conditions in the general case. The optimality conditions come in the form of a maximum principle where the absolutely continuous and the singular parts of the optimal control are characterized as pointwise minimizers of the Hamiltonian and its recession function, respectively.

We will relate problem \eqref{scp} to the following dual problem
\begin{equation}\tag{DCP}\label{dcp}
\begin{aligned}
&\minimize_{w^*\in\L^1,\eta^*\in L^1}\quad & & EI_{\tilde g^*}(w^*) + E\tilde e^*(\eta^*) + EI_h(B^T\op p) + \delta_{\R^1(D)}(B^T\op p)\\
&\st & &
\begin{cases}
\dot p_t=-A^Tp_t + w^*_t,\\
p_T=-\eta^*,
\end{cases}
\end{aligned}
\end{equation}
where $\tilde g_t(x,\omega)=g_t(x+\dot a_t,\omega)$, $\tilde e(x,\omega)=e(x+\dot a_T)$ and $a_t=\int_0^te^{(t-s)A}W_sd\mu_s$. 
We say that a normal $\F$-integrand $f$ is {\em integrable} if $f(x,\cdot)\in L^1$ for all $x\in\reals^d$. Combining Theorem~\ref{thm:ifam2} with the conjugate duality framework of \cite{roc74} yields the following. The proof is given in Section~\ref{sec:proofscp} below.

\begin{theorem}\label{thm:scp}
Assume that the optimum value is finite, $\tilde g$ and $\tilde e$ are integrable, $h$ is regular, and that $J_{h^*}(0)\in L^1$. Then $\inf\eqref{scp}=-\inf\eqref{dcp}$ and the infimum in \eqref{dcp} is attained. Moreover, $c\in\N^\infty_0$ attains the infimum in \eqref{scp} if and only if there exist $w^*\in\L^1$ such that $p_T\in L^1$ and, almost surely,
\begin{align*}
\partial h(B^T\op p) &\ni dc/d\mu\quad \mu\text{-a.e.}\\
\partial^sh(B^T\op p) &\ni dc^s/|dc|\quad |Dc^s|\text{-a.e.}\\
\partial g(\dot z) &\ni w^*\quad \mu\text{-a.e.}\\
\partial e(z_T) &\ni -p_T,
\end{align*}
where $z$ and $p$ are the corresponding solutions of the primal and dual system equations.
\end{theorem}

The optimality conditions in Theorem~\ref{thm:scp} can be written in terms of the {\em Hamiltonian}
\[
H_t(z,c,p) := g_t(z) + h^*_t(c) - p\cdot(Az+ Bc+W_t)
\]
much like in the classical Pontryagin maximum principle. Indeed, by \cite[Theorem~23.5]{roc71a}, the first two subdifferential inclusions in Theorem~\ref{thm:scp} mean that
\begin{align*}
dc/d\mu &\in \argmin_cH(\dot z,c,\op p) \quad \mu\text{-a.e.},\\
dc^s/|dc^s| &\in \argmin_cH^\infty(\dot z,c,\op p) \quad |Dc^s|\text{-a.e.},
\end{align*}
where $H^\infty(\cdot,\cdot,p):=H(\cdot,\cdot,p)^\infty$, while the third one implies that the dual state $p$ satisfies the differential inclusion
\[
\dot p\in\partial_zH(\dot z,c,p)
\]
in symmetry with the primal system equation which can be written as 
\[
\dot z\in\partial_p[-H](z,c,p).
\]
The above conditions are reminiscent of the maximum principle derived in \cite{ch94} for problems where the objective is linear in the singular part of the control. 
While the maximum principle of \cite{ch94} characterizes optimal control processes as minimizers of a certain integral functional, the above conditions give explicit pointwise characterizations for both the absolutely continuous and singular parts.

\begin{example}
Consider the problem
\begin{equation*}
\maximize\quad E\left[\int U_t(c_t)d\mu-\int D_t dc+U_T(c_T)\right]\quad\ovr c\in\N_0^\infty \text{ with }dc\ge 0,
\end{equation*}
where $U$ is a nondecreasing concave optional integrand, and $D=\op {\tilde D}$ for some nonnegative nonincreasing c\'adl\'ag process $\tilde D$. This is a finite horizon version of problem (17) in \cite[Theorem~3.1]{bk16}. 

This problem fits into \eqref{scp} with $d=1$, $A=0$, $B=1$, $W=0$, $g_t(c,\omega)=-U_t(c,\omega)$, $e=-U_T(c,\omega)$ and 
\begin{align*}
h_t(y,\omega)=
\begin{cases} 0\quad &\text{if } y\le D_t(\omega),\\
+\infty\quad &\text{otherwise.}
\end{cases}
\end{align*}
Indeed, we then have
\[
h^*_t(c,\omega)=
\begin{cases} D_t(\omega)c \quad &\text{if } c\ge 0, \\
+\infty\quad &\text{otherwise,}
\end{cases}
\]
so (up to a change of signs) \eqref{scp} reduces to the problem above. Moreover, $h$ is regular. Indeed, defining
\begin{align*}
\tilde h_t(y,\omega)=
\begin{cases} 0\quad &\text{if } y\le \tilde D_t(\omega)\\
+\infty\quad &\text{otherwise,}
\end{cases}
\end{align*}
we have
\[
\tilde h^*_t(c,\omega)=
\begin{cases} \tilde D_t(\omega)c \quad &\text{if } c\ge 0, \\
+\infty\quad &\text{otherwise}
\end{cases}
\]
and $h^*=\op(\tilde h^*)$. Since $\tilde D$ is nonincreasing and cadlag, the distance function $t\mapsto d(y,\dom \tilde h_t)$ is upper semicontinuous for every $y\in\reals$ so $\dom \tilde h_t$ is inner semicontinuous, by \cite[Proposition 5.11]{rw98}.

The dual problem \eqref{dcp} becomes
\[
\minimize\quad E\left[\int V_t(-\dot p_t)d\mu + V_T(p_T)\right]\quad\ovr \dot p\in\L^1 \text{ with }\op p\le D,
\]
where $V_t(q)=(-U_t)^*(q)$. This corresponds to (11) of \cite{bk16}. The optimality conditions can be written as
\begin{align*}
\op p &\le D,\ Dc\ge 0,\ \int (D-\op p)dc=0,\\
\dot p_t &\in \partial [-U_t](c_t)d\mu,\\
p_T &\in\partial [-U_T](c_T).
\end{align*}
Assuming $U$ is differentiable, these can be written as
\begin{align*}
\op p &\le D,\ Dc\ge 0,\ \int (D-\op p)dc=0,\\
p_t &=U'_T(c_T)+\int_t^T U'_t(c_t)d\mu
\end{align*}
which are exactly the optimality conditions in \cite[Theorem~3.1]{bk16}.

Bank and Kauppila consider the case where, almost surely, $U_t(\cdot,\omega)$ is strictly convex, differentiable on $(0,\infty)$, satisfies the Inada conditions $U'_t(0,\omega)=\infty$ and $\lim_{c\to\infty} U'_t(c,\omega)=0$, and $U(c)\in \L^1$ for every $c\in\reals_+$.  Theorem~\ref{thm:scp} establishes the existence of a dual solution in the complementary case where $U(c)\in\L^1$ for every $c\in\reals$. 
\end{example}

\subsection{Proof of Theorem~\ref{thm:scp}}\label{sec:proofscp}

Our proof is based on general results of Rockafellar~\cite{roc74} on duality and optimality in convex optimization problems that here take the form
\begin{equation}\label{p}\tag{P}
\minimize\quad F(x,u)\quad\ovr\quad x\in\N^\infty_0,
\end{equation}
where the parameter $u$ belongs to a LCTVS $U$ in separating duality with a LCTVS $Y$ and $F$ is a proper convex function on $\N^\infty_0\times U$ such that $F(x,\cdot)$ is closed for every $x\in\N^\infty_0$. The associated {\em Lagrangian}
\[
L(x,y) := \inf\{F(x,u)-\langle u,y\rangle\}
\]
is an extended real-valued function on $\N^\infty_0\times Y$, convex in $x$ and concave in $y$. Denoting the optimum value of \eqref{p} by $\varphi(u)$, we have
\[
\varphi^*(y) = -\inf_{x\in\N_0^\infty}L(x,y).
\]
The following result, obtained by combining Theorem~17 and Corollary~15A of \cite{roc74}, suffices for us.

\begin{theorem}\label{thm:cd}
Assume that the optimal value function
\[
\varphi(u) = \inf_{x\in\N^\infty_0}F(x,u)
\]
is proper and continuous on $U$. Then $\varphi=\varphi^{**}$ and an $x\in\N^\infty_0$ solves \eqref{p} if and only if there exists $y\in Y$ such that
\[
0\in\partial_x L(x,y)\quad\text{and}\quad u\in\partial_y [-L](x,y).
\]
\end{theorem}



\begin{lemma}\label{lem:A}
For any $c\in\N_0^\infty$, the system equation of \eqref{scp} has a unique solution given by 
\[
z=\A c+a
\]
where $\A c$ is the unique pathwise solution of
\[
\begin{cases}
z_0=0,\\
\dot z_t=Az_t + Bc_t.
\end{cases}
\]
Denoting $\dot\A c:=\frac{d\A c}{d\mu}$, the linear mapping $c\mapsto(\dot\A c,(\dot\A c)_T)$ from $\N^\infty_0$ to $\L^\infty\times L^\infty$ is continuous and for each $c\in\N_0^\infty$ and $(w^*,\eta^*)\in\L^1\times L^1$, we have 
\[
\langle (\dot\A c,(\dot\A c)_T),(w^*,\eta^*)\rangle = \langle -B^T\op p,Dc\rangle_{\R^1},
\]
where $p$ is the unique pathwise solution of
\[
\begin{cases}
p_T=-\eta^*,\\
\dot p_t=-A^Tp_t + w^*_t.
\end{cases}
\]
\end{lemma}

\begin{proof}
Given $c\in\N^\infty_0$ and $(w^*,\eta^*)\in\L^1\times L^1$, let $z=\A c$ and let $p$ be the corresponding solution to the dual system equation. Integration by parts gives
\begin{align*}
\langle (\dot\A c,(\dot\A c)_T),(w^*,\eta^*)\rangle &= E[\int\dot z(A^Tp+\dot p)d\mu - \dot z_T\cdot p_T]\\
&= E[\int \dot p\cdot(\dot z-Az)d\mu - (\dot z_T-Az_T)\cdot p_T]\\
&= E[\int \dot p\cdot Bcd\mu - (Bc_T)\cdot p_T]\\
&= -E\int B^Tpdc\\
&= -E\int B^T\op pdc\\
&= \langle -B^T\op p,Dc\rangle_{\R^1}.
\end{align*}
\end{proof}

Using Lemma~\ref{lem:A} we can write~\eqref{scp} as
\begin{equation}\label{scp2}
\minimize\quad  E[I_{\tilde g}(\dot\A c) + \tilde e((\dot\A c)_T) + J_{h^*}(Dc)]\quad \ovr\quad c \in\N_0^\infty,
\end{equation}
Since $\dot a$ is an optional process, $\tilde g$ and $\tilde e$ are optional and $\F$-normal integrands, respectively.

\begin{proof}[Proof of Theorem~\ref{thm:scp}]
Problem \eqref{scp2} fits the general conjugate duality framework with $X=\N^\infty_0$, $U=\L^\infty\times L^\infty$, $Y=\L^1\times L^1$ and
\[
F(c,u)=E[I_{\tilde g}(\dot\A c+w)+\tilde e((\dot\A c)_{T}+\eta)+J_{h^*}(Dc)],
\]
where $u=(w,\eta)$. Clearly
\[
\varphi(u):=\inf_{c\in\N^\infty_0}F(c,u)\le F(0,u) = E[I_{\tilde g}(w)+\tilde e(\eta) + J_{h^*}(0)].
\]
Since $\tilde g$ and $\tilde e$ are integrable, the last expression is Mackey-continuous on $U$; see \cite[Theorem~22]{roc74}. By \cite[Theorem~8]{roc74}, $\varphi$ is then Mackey-continuous as well. We may thus apply Theorem~\ref{thm:cd}. 

Denoting $y=(w^*,\eta^*)$ and using the interchange rule, the Lagrangian can be expressed as
\begin{align*}
L(c,y) &=\inf_{u\in\U}\{E\int [\tilde g(\dot\A c+w)-w\cdot w^*]d\mu+ E[\tilde e((\dot\A c)_{T}+\eta)-\eta\cdot\eta^*]\}+EJ_{h^*}(Dc)\\
&=E\int [\dot\A c\cdot w^*-\tilde g^*(w^*)]d\mu+ E[(\dot\A c)_{T}\cdot\eta^*-\tilde e^*(\eta^*)] + EJ_{h^*}(Dc)\\
&=\langle (\dot\A c,(\dot\A c)_T),(w^*,\eta^*)\rangle - EI_{\tilde g^*}(w^*) - E\tilde e^*(\eta^*) + EJ_{h^*}(Dc).
\end{align*}
By Lemma~\ref{lem:A}, $\langle (\dot\A c,(\dot\A c)_T),(w^*,\eta^*)\rangle=\langle -B^T\op p,Dc\rangle_{\R^1}$, where $p$ is the solution to the dual system equations. Theorem~\ref{thm:ifam2} now gives
\[
\varphi^*(y) = -\inf_{c\in\N^\infty_0}L(c,y) = EI_{\tilde g^*}(w^*) + E\tilde e^*(\eta^*) + EI_h(B^T\op p),
\]
so the
first claim follows from Theorem~\ref{thm:cd}.

Subdifferentiating the Lagrangian gives the optimality conditions
\begin{align*}
\partial EJ_{h^*}(Dc) - B^T\op p&\ni 0,\\
\partial EI_{\tilde g^*}(w^*) - \dot\A c&\ni 0,\\
\partial E\tilde e^*(\eta^*) - (\dot\A c)_T &\ni 0,
\end{align*}
or equivalently,
\begin{align*}
\partial(EJ_{h^*})^*(B^T\op p) &\ni Dc,\\
\partial(EI_{\tilde g^*})^*(\dot\A c) &\ni w^*,\\
\partial(E\tilde e^*)^*((\dot\A c)_T)  &\ni \eta^*.
\end{align*}
Applying the subdifferential formulas in Theorem~\ref{thm:cif} and Theorem~\ref{thm:ifam2} and recalling that $\tilde g_t(x,\omega)=g_t(x+\dot a_t,\omega)$ and $\tilde e(x,\omega)=e(x+\dot a_T)$, gives the optimality conditions in the statement.
\end{proof}

\section{Appendix}

\subsection{Normal integrands on Suslin spaces}

This section proves the claims made in Section~\ref{sec:if} concerning criteria for checking that a function $f:U\times\Xi\rightarrow\ereals$ is a normal integrand. For $U=\reals^d$, these results are well known and can be found e.g.\ in \cite{roc76} and \cite[Chapter~14]{rw98}. Various extensions exist beyond the finite-dimensional case. Below, we allow for a locally convex Suslin space $U$ which covers the function spaces studied in this paper. Note that Suslin spaces are separable since the image of a countable dense set under a continuous surjection is dense.

\begin{lemma}\label{lem:nor}
If $(\Xi,\A)$ is complete with respect to some $\sigma$-finite measure $m$, then any $\B(U)\otimes\A$-measurable function $f$ such that $f(\cdot,\xi)$ is lsc, is a normal integrand. The converse holds if $U$ is a countable union of Borel sets that are Polish spaces in the relative topology.
\end{lemma}
\begin{proof}
Since $f$ is measurable, the graph $\gph(\epi f)$ of the set-valued mapping $\xi\tos\epi f(\cdot,\xi)$ is $\F\otimes\B(U)\otimes\B(\reals)$-measurable, where $\B(U)\otimes\B(\reals)=\B(U\times\reals)$ \cite[Lemma 6.4.2]{bog7}. The space $U\times\reals$ is also Suslin \cite[Lemma 6.6.5]{bog7}. For any open $O\subset U\times\reals$, the set $\gph(\epi f)\cap (\Omega\times O)$ is measurable, so by the projection theorem \cite[Theorem III.23]{cv77}, $\{\omega \mid \exists (u,\alpha): (\omega,u,\alpha)\in\gph(\epi f)\cap (\Omega\times O)\}$ belongs to $\F$.

To prove the converse, denote by $P^\nu$ the Borel sets in question. We have that $\B(P^\nu)$ coincides with $\B(U)\cap P^\nu$, so $f$ is jointly measurable if and only if its restriction $f^\nu:=f_{|P^\nu\times\Omega}$ is jointly measurable for each $\nu$. 
Since $\epi f^\nu$ is a measurable closed-valued mapping from $\Omega$ to $P^\nu\times\reals$, its graph is measurable by \cite[Theorem III.30]{cv77}, and thus, by \cite[Lemma 7]{val75}, $f^\nu$ is jointly measurable.
\end{proof}


From now on we assume that $U$ is a countable union of Borel sets that are Polish spaces in the relative topology. A function $f$ satisfying the assumptions of the following proposition is known as a {\em Carath\'eodory integrand}. If $(\Xi,\A)$ were complete w.r.t.\ some $\sigma$-finite measure $m$, the result would follow from \cite[Lemma 1.2.3]{crv4} and Lemma~\ref{lem:nor}.

\begin{proposition}\label{prop:cni}
If $f:U\times\Xi\rightarrow\reals$ is such that $f(\cdot,\xi)$ is continuous for every $\xi$ and $f(u,\cdot)$ is measurable for every $u\in U$, then $f$ is a normal integrand.
\end{proposition}

\begin{proof}
Let $\{u^\nu\mid \nu\in\mathbb N\}$ be a dense set in $U$ and define $\alpha^{\nu,q}(\xi)=f(u^\nu,\xi)+q$, where $q\in\mathbb Q_+$. Since $f(\cdot,\xi)$ is continuous, the set $\hat O=\{(u,\alpha)\mid f(u,\xi)<\alpha \}$ is open. For any $(u,\alpha)\in\epi f(\cdot,\xi)$ and for any open neighborhood $O$ of $(u,\alpha)$, $O\cap\hat O$ is open and nonempty, and there exists $(u^\nu,\alpha^{\nu,q})\in O\cap\hat O$, i.e.,  $\{(u^\nu,\alpha^{\nu,q}(\xi)\mid \nu\in\mathbb N, q\in\mathbb Q\}$ is dense in $\epi f(\cdot,\xi)$. Thus for any open set $O\in U\times\reals$,
\begin{align*}
 \{\xi \mid \epi f(\cdot,\xi)\cap O\neq\emptyset\}=\bigcup_{\nu,q}\{\xi \mid (u^\nu,\alpha^{\nu,q}(\xi))\in O\}
\end{align*}
is measurable. 
\end{proof}

Given a proper lower semicontinuous function $f$ and a nonnegative scalar $\alpha$, we define
\[
(\alpha f)(x):=
\begin{cases}
\alpha f(x) & \text{if $\alpha>0$},\\
\delta_{\cl\dom f}(x) & \text{if $\alpha=0$}.
\end{cases}
\]

\begin{lemma}\label{lem:oni}
Let $f$, $f^i$, $i=1,\dots, n$ be proper normal integrands on $U$ and assume that $(\Xi,\A)$ is complete w.r.t.\ a $\sigma$-finite measure or $U=\reals^d$, then the functions
\begin{enumerate}
\item $(u,\xi)\mapsto\alpha(\xi)f(u,\xi)$, where $\alpha\in L^0(\Xi,\reals_+)$,
\item $(u,\xi)\mapsto\sum_{i=1}^nf^i(u,\xi)$,
\item
$f^\infty$ defined by $f^\infty(\cdot,\xi)=f(\cdot,\xi)^\infty$, 
\item $f^*$ defined by $f^*(\cdot,\xi)=f(\cdot,\xi)^*$, where $Y$ is a Suslin space in separating duality with $U$
\end{enumerate}
are normal integrands.
\end{lemma}
\begin{proof}
For $U=\reals^d$, proofs can be found in \cite[Chapter 14]{rw98} except for part 1. In this case, using properness of $f(\cdot,\xi)$, it can be shown that $\alpha(\xi)f(\cdot,\xi)$ is the epigraphical limit of $\alpha^\nu f(\cdot,\xi)$ as $\alpha^\nu\searrow\alpha(\xi)$ (see \cite[Chapter 7]{rw98}), so the result follows from \cite[Proposition 14.53]{rw98}.

When $(\Xi,\A)$ is complete, it suffices to verify, by Lemma~\ref{lem:nor}, that the functions are jointly measurable and lower semicontinuous in the second argument. For 2, this follows from the fact that sums of measurable/lsc functions is again measurable/lsc. For 3, it suffices to note that $f^\infty$ is a pointwise increasing limit of jointly measurable functions that are lower semicontinuous in the second argument. Part 4 is from \cite[Lemma 8]{val75}. To prove 1, note first that $\{(u,\xi)\mid f(u,\xi)<\infty\}$ is measurable, so one may proceed as in the proof of Lemma~\ref{lem:nor} to show that $\dom f$ is a measurable set-valued mapping. It follows that $\cl\dom f$ is a measurable mapping as well, so $\delta_{\cl\dom f}$ is a normal integrand and thus jointly measurable by Lemma~\ref{lem:nor}. Since $\alpha f=\one_{\alpha=0}\delta_{cl\dom f}+\one_{\alpha>0}\alpha f$, we get, using part 2, that $f$ is jointly measurable.
\end{proof}

\subsection{Integral functionals as normal integrands}

Let $U$ be a Suslin subspace of Borel measurable functions on $[0,T]$ and consider the random integral functional $I_h:U\times\Omega\to\ereals$ defined scenariowise by
\[
I_h(u,\omega):=I_{h(\cdot,\omega)},
\]
where $h$ and $\mu$ are as in Section~\ref{sec:ifni}. That is, $\mu$ is a nonnegative random Radon measure with full support almost surely and $h$ is a convex normal integrand on $\reals^d\times\Xi$, where $\Xi=\Omega\times[0,T]$ is equipped with the product sigma-algebra $\F\otimes\B([0,T])$.
The following result was used in Lemma~\ref{lem:nif} for $U=C$.
\begin{theorem}\label{thm:mif}
Assume that $U$ is a countable union of Borel sets that are Polish spaces in the relative topology. The function $I_h$ is $\B(U)\otimes\F$-measurable under either of the following conditions:
\begin{enumerate}
\item[(a)] The sequential convergence in $U$ implies pointwise convergence outside a countable set and point-evaluations in $U$ are measurable.
\item[(b)] The topology on $U$ be finer than the topology of convergence in $\mu(\omega)$-measure almost surely.
\end{enumerate}
In such cases, $I_h$ is a normal integrand on $U$ whenever $I_h(\cdot,\omega)$ is lower semicontinuous almost surely.
\end{theorem}
\begin{proof}
We only give proof for the first set of conditions, the second case is similar. We may assume without loss of generality that $h$ is bounded from below. Indeed, If $I_{h^\alpha}$ is $\B(U)\otimes\F$-measurable for $h^\alpha=\sup\{h,\alpha\}$ and $\alpha<0$, then, by the monotone convergence theorem (recall that by convention $I_h(u,\omega)=+\infty$ unless $I_{h^+}(u,\omega)<+\infty$),
\[
I_h(u,\omega)=\lim_{\alpha\rightarrow -\infty}I_{h^\alpha}(u,\omega),
\]
and $I_h$ is $\B(U)\otimes\F$-measurable as well. Assume first that $\mu$ is an atomless random measure.

Case 1: Assume that $\alpha\leq h_t(u,\omega)\leq \gamma$ for all $(\omega,t,u)$ and that $h_t(\cdot,\omega)$ is continuous for all $(\omega,t)$. By the dominated convergence theorem and continuity of $h_t(\cdot,\omega)$,  $I_h(\cdot,\omega)$ is continuous in $\mu(\omega)$-measure and thus continuous in $U$. For every $u\in U$, $I_h(u,\cdot)$ is measurable, since $(t,\omega)\mapsto h_t(u_t,\omega)$ is measurable (being a composition of measurable mappings) and $\mu$ is a random Radon measure. By Proposition~\ref{prop:cni} and Lemma~\ref{lem:nor}, $I_h$ is thus $\B(U)\otimes\F$-measurable.

Case 2: Assume that $\alpha\leq h_t(u,\omega)$ for all $(\omega,t,u)$ and that $h_t(\cdot,\omega)$ is continuous for all $(\omega,t)$. By Case 1, for $h^\gamma=\min\{h,\gamma\}$, $I_{h^\gamma}$ is $\B(U)\otimes\F$-measurable. By Monotone convergence theorem,
\[
I_h(u,\omega)=\lim_{\gamma\rightarrow\infty}I_{h^\gamma}(u,\omega),
\]
and therefore $I_h$ is $\B(U)\otimes\F$-measurable.

Case 3: Assume that $\alpha\leq h_t(u,\omega)$ for all $(\omega,t,u)$. By \cite[p. 665]{rw98},
\[
h^\lambda_t(u,\omega)=\inf_{u'}\{h_t(u',\omega)+\frac{1}{\lambda}|u-u'|\}
\]
is a normal integrand, so $I_{h^\lambda}$ is $\B(U)\otimes\F$-measurable by Case 2.
By \cite[Example 9.11]{rw98}, $h^\lambda\upto h$ as $\lambda\to\infty$ so, by the monotone convergence theorem,
\[
I_h(u,\omega)=\lim_{\lambda\rightarrow 0}I_{h^\lambda}(u,\omega).
\]
Thus, $I_h$ is $\B(U)\otimes\F$-measurable.

Consider now the case of a general nonnegative random Radon measure $\mu$. By \cite[Theorem 3.42]{hwy92}, $\mu=\mu^c+\mu^{d}$, where $\mu^c$ is atomless and $\mu^{d}$ is supported on the union $\bigcup_i[\tau^i]$ of graphs of some random times $\tau^i$. We have the decomposition $I_h=I_h^c+I_h^d$ where the integral functionals are defined with respect to $\mu^c$ and $\mu^d$, respectively. Here $I_h^c$ is measurable by the first part. We have
\[
I_h^d(u,\omega) = \sum_i h_{\tau^i(\omega)}(u_{\tau^i(\omega)},\omega) \mu^d({\tau^i(\omega)}).
\]
By the monotone convergence theorem, it suffices to prove that each term in the sum defines a $\B(U)\otimes\F$-measurable function. The maps $(u,\omega)\mapsto (u,\tau(\omega),\omega)$, $(u,t,\omega)\mapsto (u_t,t,\omega)$ and $(x,t,\omega)\mapsto h_t(x,\omega)$ are measurable, so $(u,\omega)\mapsto h_{\tau(\omega)}(u_{\tau(\omega)},\omega)$ is a composition of measurable mappings. 

The last claim follows from Lemma~\ref{lem:nor}. 
\end{proof}

\subsection{Projections of extended-real valued processes}

\begin{lemma}\label{lem:op0}
Let $v$ be a measurable extended real-valued process and let $\mu$ be a nonnegative random Radon measure. Then 
\[
\omega\mapsto \int  v_t(\omega)d\mu_t(\omega) := \int  v^+_t(\omega)d\mu_t(\omega)-\int v^-d\mu_t(\omega)
\]
is $\F$-measurable.
\end{lemma}
\begin{proof}
When $v$ is finite-valued and $\int |v_t|d\mu_t<\infty$ a.s., the claim follows from standard monotone class arguments. For a nonnegative $v$, the claim then follows from the monotone convergence theorem. For an arbitrary $v$, the integral is thus a sum of $\F$-measurable extended real-valued random variables.
\end{proof}

We will need the notion of an optional projection of an extended real-valued processes; see \cite{kp16}. For a nonnegative real-valued process $v$, there exists a unique optional process $\op v$ satisfying \eqref{eq:op}; see \cite[Theorem VI.43]{dm82}. The monotone convergence theorem then gives the existence of a unique optional projection for a nonnegative extended real-valued process as well. For an extended real-valued stochastic process $v$ with $\T$-integrable $v^+$ or $v^-$, we define $\op v= \op (v^+) - \op (v^-)$. 

\begin{lemma}\label{lem:op}
Let $\mu$ be a nonnegative optional Radon measure and $v$ an extended real-valued process such that $v^+$ or $v^-$ is $\T$-integrable.  If $E\int v^-d\mu<\infty$, then 
\[
E\int vd\mu=E\int \op vd\mu.
\]
\end{lemma}
\begin{proof}
By \cite[Theorem VI.57]{dm82},
\[
E\int vd\mu=E\int \op vd\mu
\]
for all nonnegative real-valued processes $v$. For nonnegative extended real-valued process, the expression is then valid by the monotone convergence theorem. As to the general case, 
\begin{align*}
E\int vd\mu &= E\int \left[v^+-v^-\right] d\mu\\
&= E\int v^+d\mu-E\int v^-d\mu\\
&= E\int \op v^+d\mu-E\int \op v^-d\mu\\
&=  E\int \left[\op v^+-\op v^-\right]d\mu\\
&= E\int \op vd\mu,
\end{align*}
where the second equality holds since $E\int v^-d\mu<\infty$ and the fourth since $E\int \op v^-d\mu<\infty$, by \cite[Theorem VI.57]{dm82}.
\end{proof}

\subsection{Proof of Theorem~\ref{thm:ifam2}}

\begin{lemma}\label{lem:iop}
Let $\mu$ be an optional random Radon measure and $h$ a convex normal integrand such that $h^*(\bar x)^+$ and $h(\bar y)^+$ are $\T$-integrable for some optional $\bar x$ and $\T$-integrable $\bar y$. If $x$ is an optional process such that $E\int h^*(x)^-d\mu<\infty$, then 
\[
E\int  h^*(x)d\mu=E\int  \op(h^*)(x)d\mu.
\]
\end{lemma}

\begin{proof}
Let $A=\{|x|\ge M\}$ for some strictly positive $M\in\reals$. We have
\[
E\int  h^*(x)d\mu= E\int  [\one_Ah^*(x)]d\mu+E\int  [\one_{A^C}h^*(x)]d\mu,
\]
since the negative parts of both terms are integrable. 

Let $\lambda=\one_A|x|^{-1}$, $\alpha=\one_A x/|x|$,  and $d\hat\mu=|x|d\mu$ so that
\[
E\int [\one_Ah^*(x)]d\mu=E\int \hat h^*(\alpha,\lambda)d\hat\mu
\]
for $\hat h_t(\beta,\eta):=\delta_{\epi h_t}(\beta,-\eta)$. Indeed, by \cite[Corollary~13.5.1]{roc70a}, we have
\[
\hat h^*_t(\alpha,\lambda,\omega)=\begin{cases}
\lambda h^*_t(\alpha/\lambda,\omega)\quad&\text{if }\lambda>0,\\
(h^*_t)^\infty(\alpha,\omega)\quad&\text{if }\lambda=0,\\
+\infty\quad&\text{otherwise.}
\end{cases}
\]
Since $\lambda$ and $\alpha$ are bounded optional processes, we get from Lemma~\ref{lem:op} and \eqref{eq:opni} that 
\[
E\int \hat h^*(\alpha,\lambda)d\hat\mu= E\int \op(\hat h^*)(\alpha,\lambda)]d\hat\mu.
\]
It is not difficult to verify from the definitions that $\op{\epi{\hat h}}=\epi\delta_{\op{\epi h}}$, so $\op (\hat h^*) (\alpha,\lambda)= \lambda\op (h^*)(\alpha/\lambda)$ by \cite[Theorem~10]{kp16}. Thus
\[
E\int  [\one_Ah^*(x)]d\mu= E\int  [\one_A\op(h^*)(x)]d\mu.
\]
As to the second term, $\one_{A^C}x$ is bounded, so $\one_{A^C}h^*(x)^-$ is $\T$-integrable by the Fenchel inequality, and thus, by Lemma~\ref{lem:op} and \eqref{eq:opni},
\[
E\int  [\one_{A^C}h^*(x)]d\mu= E\int  [\one_{A^C}\op(h^*)(x)]d\mu.
\]
Thus
\[
E\int  h^*(x)d\mu= E\int  [\one_A\op(h^*)(x)]d\mu+E\int  [\one_{A^C}\op(h^*)(x)]d\mu,
\]
which finishes the proof since the negative parts of both terms are again integrable.
\end{proof}

\begin{lemma}\label{lem:propi}
Let $h$ be regular and $\tilde h$ as in Definition~\ref{def:ri}. Then $\int \tilde h(y)^-d\mu$, $\int \tilde h^*(d\theta^a/d\mu)^-d\mu$ and $\int (\tilde h^*)^\infty(d\theta/d|\theta^s|)^-d|\theta^s|$ are integrable for every $y\in L^1(C)$ and $\theta\in L^\infty(M)$. In particular, $EI_{\tilde h}$ and $EJ_{\tilde h^*}$ are proper.
\end{lemma}
\begin{proof}
We define $\bar\theta\in\M^\infty$ by $d\bar\theta^a/d\mu=\bar x$ and $\bar\theta^s=0$, where $\bar x$ is from Definition~ \ref{def:ri}. The first lower bound in Definition~\ref{def:ri} implies
\begin{align*}
E\int \tilde h(y)^-d\mu &\le E\int [(d\bar\theta^a/d\mu\cdot y)+\alpha]^+d\mu\le E\left[\|\bar\theta\|\| y\|\right]+ E\int\alpha d\mu.
\end{align*}
The other terms are handled similarly, where, for the recession function, the latter lower bound implies $(\tilde h^*)^\infty(x)\ge x\cdot \bar v$. The bounds also give that $EI_{\tilde h}(\bar v)\le E\int \alpha d\mu$ and $EJ_{\tilde h^*}(\bar\theta)\le E\int \alpha d\mu$, so $EI_{\tilde h}$ and $EJ_{\tilde h^*}$ are proper.
\end{proof}

\begin{lemma}\label{lem:opJh*}
Let $h$ be regular and $\tilde h$ as in Definition~\ref{def:ri}. Then $J_{h^*}(\theta)$ is $\F$-measurable and $EJ_{h^*}(\theta)=EJ_{\tilde h^*}(\theta)$ for every $\theta\in\M^\infty$.
\end{lemma}

\begin{proof}
By \cite[Theorem 5.14]{hwy92}, the Radon--Nikodym densities $d\theta^a/d\mu$ and $d\theta^s/d|\theta^s|$ are optional processes, so $h^*((d\theta^a/d\mu))$ and $(h^*)^\infty((d\theta^s/d|\theta^s|)$ are optional as well. By Lemma~\ref{lem:op0}, $J_{h^*}(\theta)$ is thus $\F$-measurable and $EJ_{h^*}$ is well-defined. By Lemma~\ref{lem:propi}, $\int \tilde h^*(d\theta^a/d\mu)^-d\mu$ and $\int (\tilde h^*)^\infty(d\theta^s/d|\theta^s|)^-d|\theta^s|$ are integrable, so
\begin{align*}
EJ_{\tilde h^*}(\theta) &=E\left[\int \tilde h^*(d\theta^a/d\mu)d\mu + \int (\tilde h^*)^\infty(d\theta^s/d|\theta^s|)d|\theta^s|\right]\\
&= E\int \tilde h^*(d\theta^a/d\mu)d\mu + E\int (\tilde h^*)^\infty(d\theta^s/d|\theta^s|)d|\theta^s|\\
&= E\int \op[\tilde h^*(d\theta^a/d\mu)]d\mu + E\int \op[(\tilde h^*)^\infty(d\theta^s/d|\theta^s|)]d|\theta^s|,
\end{align*}
where the last equality follows from Lemma~\ref{lem:op}. The lower bounds in Theorem~\ref{thm:ifam2} give that $\tilde h(\bar v)^+$ and $\tilde h^*(\bar x)^+$ are $\T$-integrable, so \cite[Theorem 7]{kp16} implies that  $\op((\tilde h^*)^\infty)=(\op (\tilde h^*))^\infty$. Consequently, $EJ_{\tilde h^*}(\theta)=EJ_{h^*}(\theta)$ by Lemma~\ref{lem:iop}.
\end{proof}



\noindent 
{\it Proof of Theorem~\ref{thm:ifam2}.}
We get from the lower bounds in Definition~ \ref{def:ri} that $\tilde h(\bar v)^+$ and $\tilde h^*(\bar x)^+$ are $\T$-integrable, so, by \cite[Lemma~3 and Theorem~8]{kp16}, we have
\begin{align*}
h(v) &\ge v\cdot \bar x-\op{\alpha},\\
h^*(x) &\ge \op{\bar v}\cdot x-\op{\alpha}.
\end{align*}
Thus $EI_h(\op{\bar v})\le E\int\op\alpha d\mu$, where the right side is finite by Lemma~\ref{lem:op}. Similarly, $EJ_{h^*}(\bar\theta)$ is finite for $\bar\theta\in\M^\infty$ defined by $d\bar\theta^a/d\mu=\bar x$ and $\bar\theta^s=0$. The lower bounds also imply that $EI_h$ and $EJ_{h^*}$ never take the value $-\infty$ on $\R^1$ and $\M^\infty$. Thus $EI_h$ and $EJ_{h^*}$ are proper.

To prove that $(EI_h+\delta_{\R^1(D)})^*=EJ_{h^*}$, let $v\in\dom EI_h$ and $\theta\in\dom EJ_{h^*}$. Since $\delta_{D_t}^*=(h_t^*)^\infty$, we have, almost surely, the Fenchel inequalities
\begin{align}\label{eq:fen}
\begin{split}
h(v)+h^*(d\theta^a/d\mu)&\ge v\cdot(d\theta^a/d\mu)\quad\mu\text{-a.e.}\\
(h^*)^\infty(d\theta^s/d|\theta^s|) &\ge v\cdot (d\theta^s/d|\theta^s|)\quad|\theta|^s\text{-a.e.},
\end{split}
\end{align}
so $I_h(v)+\delta_{\R^1(D)}(v)+J_{h^*}(\theta)\ge \langle v,\theta\rangle$ a.s.\ and thus, $(EI_h+\delta_{\R^1(D)})^*\le EJ_{h^*}$.

To prove the opposite inequality, we assume first that $\bar x=0$, so that $\tilde h$ is bounded from below. Since $\tilde h(\bar v)^+$ and $\tilde h^*(0)^+$ are $\T$-integrable, \cite[Theorem 9]{kp16} implies that $h(\op y) \le \op[\tilde h(y)]$ for every $y\in L^1(C)$. Likewise, since $\delta_{D}=((h^*)^\infty)^*$, \cite[Theorem 9]{kp16} gives $\delta_{D}(\op y)\le\op\delta_{\cl\dom\tilde h}(y)$ for every $y\in L^1(C)$. Thus, by Lemma~\ref{lem:op}, $EI_h(\op y)+\delta_{\R^1(D)}(\op y)\le E[I_{\tilde h}(y)+\delta_{\cl\dom\tilde h}(y)]$, so 
\begin{align*}
(EI_h(v)+\delta_{\R^1(D)})^*(\theta)&=\sup_{v\in\R^1}\{\langle v,\theta\rangle-[EI_h(v) + \delta_{\R^1(D)}](v)]\}\\
&=\sup_{y\in L^1(C)}\{\langle y,\theta\rangle - [EI_{h}(\op y) + \delta_{\R^1(D)}](\op y)]\}\\
&\ge\sup_{y\in L^1(C)}\{\langle y,\theta\rangle-E[I_{\tilde h}(y)+\delta_{\cl\dom \tilde h}(y)]\}\\
&=E\sup_{y\in C}\{\int yd\theta - I_{\tilde h}(y)-\delta_{\cl\dom \tilde h}(y)\}\\
&=EJ_{\tilde h^*}(\theta)\\
&=EJ_{h^*}(\theta),
\end{align*}
where the last two lines follow from Theorem~\ref{thm:1} and Lemma~\ref{lem:opJh*}, respectively.

Now, let $\bar x$ be arbitrary optional with $\int|\bar x|d\mu\in L^\infty$ and define $\bar\theta\in\M^\infty$ by $d\bar\theta/d\mu=\bar x$ and $\bar\theta^s=0$. We have
\begin{align*}
(EI_h+\delta_{\R^1(D)})^*(\theta)&=\sup_{v\in\R^1}\{\langle v,\theta\rangle - [EI_h+\delta_{\R^1(D)}](v)]\}\\
&=\sup_{v\in\R^1}\{\langle v,\theta-\bar\theta\rangle-[EI_{\bar h}+\delta_{\R^1(D)}](v)]\},
\end{align*}
where $\bar h_t(v,\omega) :=h_t(v,\omega)-\bar x_t(\omega)\cdot v$. It suffices to show that $\bar h$ is regular in the sense of Definition~\ref{def:ri} with $\bar x=0$ since then, by the previous two paragraphs,
\begin{align*}
(EI_h+\delta_{\R^1(D)})^*(\theta)&=EJ_{\bar h^*}(\theta-\bar\theta).
\end{align*}
Indeed, since  $\bar h^*_t(x,\omega)=h^*_t(x+\bar x_t(\omega),\omega)$ and $(\bar h^*_t)^\infty(x,\omega)=(h^*_t)^\infty(x,\omega)$, the right side equals $EJ_{h^*}(\theta)$. Defining $\hat h_t(v,\omega) :=\tilde h_t(v,\omega)-\bar x_t(\omega)\cdot v$ we get
\[
\hat h^*_t(x,\omega) =\tilde h^*_t(x+\bar x_t(\omega),\omega),
\]
so \cite[Corollary~3]{kp16} implies $\op(\hat h^*)_t(x,\omega)=\bar h^*_t(x,\omega)$. It is now easy to verify that $\bar h$ satisfies Definition~\ref{def:ri} with $\bar x=0$. Thus $(EI_h+\delta_{\R^1(D)})^*=EJ_{h^*}$. As to the subdifferential, we have $\theta\in\partial(EI_h+\delta_{\R^1(D)})(y)$ if and only if $(EI_h+\delta_{\R^1(D)})(y)+EJ_ {h^*}(\theta)=\langle y,\theta\rangle$, which is equivalent to having equalities in \eqref{eq:fen}, which, in turn, is equivalent to the subdifferential conditions in the statement.

It remains to show that $EI_h+\delta_{\R^1(D)}$ is lower semicontinuous. If $v^\nu\to v$ in $\R^1$, we have (by passing to a subsequence if neccessary) $v^\nu\to v$ in the sup-norm almost surely\ (see \cite[p.~82-83]{dm82}) so, by Fatou's lemma, the Fenchel inequalities \eqref{eq:fen} imply that
\[
\liminf [EI_h(v^\nu)+\delta_{\R^1(D)}(v^\nu)-\langle v^\nu,\theta\rangle] \ge [EI_h(v)+\delta_{\R^1 (D)}(v^\nu)-\langle v,\theta\rangle],
\]
and thus, $\liminf [EI_h(v^\nu)+\delta_{\R^1(D)}(v^\nu)]\ge [EI_h(v)+\delta_{\R^1(D)}(v^\nu)]$.$\hfill\Box$

\bibliographystyle{alpha}
\bibliography{sp}

\end{document}